\documentclass[12pt]{amsart}

\usepackage{amsmath, amsthm, amssymb, enumerate, amscd}
\usepackage[mathscr]{eucal}
\usepackage{pgf,tikz}
\usepackage{tikz, tikz-cd}
\usetikzlibrary{matrix,decorations.pathmorphing,arrows}

\theoremstyle{plain}
\newtheorem{theorem}{Theorem}[section]
\newtheorem{proposition}[theorem]{Proposition}

\theoremstyle{definition}

\newcommand{\N}{{\mathbb{N}}}

\newcommand{\Q}{{\mathbb{Q}}}

\begin{document}

\title[An atomic non-ACCP domain]{A new simple example of an atomic domain which is not ACCP}

\thanks{$\ast$ the corresponding author}
\author{Hamid Kulosman$^\ast$}
\address{Department of Mathematics\\ 
University of Louisville\\
Louisville, KY 40292, USA}
\email{hamid.kulosman@louisville.edu}



\date{}

\begin{abstract} 
We give a new simple example of an atomic domain which is not ACCP. Our example is a monoid domain $F[X;M]$, where $F$ is a field and $M$ is a submonoid of the additive monoid of nonnegative rational numbers. 
\end{abstract}

\maketitle

\section{Introduction}\label{intro}
It is well-known that in an ACCP domain (i.e., a domain in which every increasing sequence of principal ideals is stationary) every non-zero non-unit element can be written as a finite product of irreducible elements (also called atoms). ``Somewhat surprisingly'', to quote \cite{aaz}, ``the converse is not true; an atomic domain need not satisfy ACCP, but examples are hard to come by.'' The first such example is due to A.~Grams in \cite{grams}, where she also mentioned that P.~M.~Cohn conjectured that another domain is atomic but not ACCP. (Note that in \cite{c} Cohn assumed that these two notions are equivalent.) A new type of examples was given by A.~Zaks in \cite{zaks}, where he also proved Cohn's conjecture. The goal of this note is to give a new simple example of an atomic non-ACCP domain. Our example is a monoid domain $F[X;M]$, where $F$ is a field and $M$ is a submonoid of the additive monoid of nonnegative rational numbers. We would like to mention that, in general,  it is much easier to work with the localization of $F[X;M]$ with respect to its maximal ideal consisting of ``polynomials'' with zero constant term, than with $F[X;M]$ itself, nevertheless our example is fairly simple.

\section{Notation and preliminaries}\label{notation_prelims}
We begin by recalling some definitions and statements. All the notions that we use  but not define in this paper, as well as the definitions and statements for which we do not specify the source,  can be found in the classical reference books \cite{c_book} by P.~M.~Cohn, \cite{gil} by R.~Gilmer, \cite{k} by I.~Kaplansky, and \cite{n} by D.~G.~Northcott, and in the paper \cite{gk}.

\medskip
We use $\subseteq$ to denote inclusion and $\subset$ to denote strict inclusion. We denote $\N=\{1,2,\dots\}$, $\N_0=\{0,1,2,\dots\}$, and $\Q_+=\{q\in \Q\;|\;q\ge 0\}$.

\medskip
An {\it integral domain} is a commutative ring $R\ne \{0\}$ (with multiplicative identity $1$) such that for any $x,y\in R$, $xy=0$ implies $x=0$ or $y=0$.  All the rings that we use in thsi apper are assumed to be commutative and with multiplicative identity $1$. In fact, all of them will be integral domains.
An element $a\in R$ is said to be a {\it unit} of $R$ if it has a multiplicative inverse in $R$. A non-zero non-unit element $a\in R$ is said to be {\it irreducible} (and called an {\it atom}) of $R$ if  $a=bc$ $(b,c\in R$) implies that at least one of the elements $b,c$ is a unit. Two elements $a,b\in R$ are said to be {\it associates} if $a=ub$, where $u$ is a unit. We then write $a\sim b$. An integral domain $R$  is said to be {\it atomic} if every non-zero non-unit element of $R$ can be written as a finite product of atoms. An integral domain $R$ is said to be an {\it ACCP domain} if every increasing sequence 
\[(a_1) \subseteq (a_2) \subseteq (a_3)\subseteq \dots\]
of principal ideals of $R$ is stationary, meaning that $(a_{n})=(a_{n+1})=(a_{n+2})=\dots$ for some $n$.

\medskip
Many notions related to factorization in an integral domain can be already defined in the underlying multiplicative monoid $(R,\cdot)$. Hence we can generalize them by defining them in an arbitrary commutative monoid with zero $(M,\cdot)$, and then also in an arbitrary additive monoid with infinity $(M,+)$ by just translating the terminology from the multiplicative one to the additive one. The next definitions illustrate this point of view.

\medskip
A {\it commutative monoid with infininity}, written additively, is a non\-emp\-ty set $M$ with an operation $+:M\times M\to M$ which is associative, commutative, has an identity element called {\it zero} ( i.e., an element $0\in M$ such that $a+0=a$ for every $a\in M$), and has an {\it infinity element} (i.e., an element $\infty\in M$ such that $x+\infty=\infty$ for all $x\in M$).  The infinity element corresponds to the zero element of multiplicative monoids. All the monoids with infinity used in this paper are assumed to be commutative and written additively. From now on we call them just {\it monoids}.  (They are studied, for example,  in \cite{k}.) A nonempty subset $I$ of a monoid $M$ is called an {\it ideal} of $M$ if $M+I=I$, i.e., if for every $a\in I$, $M+a\subseteq I$. (Here $S_1+S_2=\{x+y\;:\; x\in S_1,\, y\in S_2\}$ for any two subsets $S_1, S_2$ of $M$.) Note that, in particular, $\{\infty\}$ is the smallest and $M$ the biggest ideal of $M$. An ideal $I$ of $M$ is said to be {\it principal} if there is an element $a\in I$ such that $I=M+a$. We then write $I=(a)$. An element $a\in M$ is said to be a {\it unit} of $M$ if it has an additive inverse in $M$.  Two elements $a,b\in M$ are said to be {\it associates} if $a=u+b$, where $u$ is a unit. We then write $a\sim b$. A non-unit element $a\in M$ is said to be {\it irreducible} (and called an {\it atom}) of $M$ if  $a=b+c$ $(b,c\in M)$ implies that at least one of the elements $b,c$ is a unit.  A monoid $M$  is said to be {\it atomic} if every non-zero non-unit element of $M$ can be written as a finite sum of atoms. A monoid $M$ is said to be an {\it ACCP monoid} if every increasing sequence 
\[(a_1) \subseteq (a_2) \subseteq (a_3)\subseteq \dots\]
of principal ideals of $M$ is stationary.

\medskip
As we said in Introduction, it is well-known that {\it an ACCP domain is atomic}. This is in fact a property of the underlying multiplicative monoid of the domain, so the analogous statement holds for monoids as well. The converse is not true, as we are going to see in our theorem in the next section.

\medskip
If $M$ is a  monoid and $F$ is a field, we will consider the {\it monoid ring} $F[X;M]$ associated with $M$. It consists of the polynomial expressions (also called polynomials)
\[f=a_0X^{\alpha_0}+a_1X^{\alpha_1}+\dots +a_nX^{\alpha_n},\]
where $n\ge 0$, \,$a_i\in F$, \,and $\alpha_i\in M$ $(i=0,1,\dots, n)$. In all the situations in which we will be considering $F[X;M]$, $M$ is a submonoid of the monoid $\Q_+$. Then the monoid ring $F[X;M]$ is an integral domain (we will say a {\it monoid domain}), as it is easy to see, and the units of $F[X;M]$ are precisely the elements of $F$.  

\section{The example}\label{example}
\begin{proposition}\label{ACCP_iff_ACCP}
Let $M$ be a submonoid of $\Q_+$. Then the monoid domain $F[X;M]$ is ACCP if and only if the monoid $M$ is ACCP. 
\end{proposition}

\begin{proof}
Suppose $F[X;M]$ is ACCP. Consider an increasing sequence 
\[(a_1) \subseteq (a_2) \subseteq (a_3) \subseteq \dots\]
of principal ideals of $M$.
It corresponds to it an increasing sequence
\[(X^{a_1}) \subseteq (X^{a_2}) \subseteq (X^{a_3}) \subseteq \dots\] 
of principal ideals of $F[X;M]$. It is stationary since $F[X;M]$ is ACCP, hence there is an $n\in\N$ such that
\[(X^{a_n})=(X^{a_{n+1}})=(X^{a_{n+2}})=\dots\]
Hence
\[X^{a_n} \sim X^{a_{n+1}} \sim X^{a_{n+2}} \sim \dots\]
Since the units of $F[X;M]$ are precisely the elements of $F$, we have
\[a_n=a_{n+1}=a_{n+2}=\dots,\]
hence
\[(a_n) = (a_{n+1}) = (a_{n+2}) = \dots\]
Hence $M$ is ACCP.

Conversely, suppose that $M$ is ACCP. Consider an increasing sequence 
\begin{equation}\label{ideals_of_FXM}
(f_1)\subseteq (f_2) \subseteq (f_3) \subseteq \dots
\end{equation}
of principal ideals of $F[X;M]$. For each $i=1,2,3,\dots$ let
\[f_i=a^i_1X^{\alpha^i_1}+a^i_2X^{\alpha^i_2}+ \dots + a^i_{n_i}X^{\alpha^i_{n_i}}\]
with 
\[\alpha^i_1>\alpha^i_2>\alpha^i_3>\dots\]
For each $i=1,2,3,\dots$ there is a $g_i\in F[X;M]$ such that
\begin{equation}\label{f_i_g_i}
f_i=f_{i+1}g_i.
\end{equation}
If $\deg(g_i)=\beta_i$ for $i=1,2,3,\dots$, then
\begin{equation}\label{alpha_beta}
\alpha^i_1=\alpha^{i+1}_1+\beta^i_1,
\end{equation}
hence we have an increasing sequence 
\[(\alpha^1_1)\subseteq (\alpha^2_1)\subseteq (\alpha^3_1)\subseteq\dots\]
of principal ideals of $M$. Since $M$ is ACCP, this sequence is stationary, hence
\[\alpha^n_1 \sim \alpha^{n+1}_1 \sim \alpha^{n+2}_1 \sim\dots,\]
hence (since $0$ is the only invertible element of $M$)
\[\alpha^n_1 = \alpha^{n+1}_1 = \alpha^{n+2}_1 = \dots\]
It now follows from (\ref{alpha_beta})
that
\[0=\beta^n_1=\beta^{n+1}_1=\beta^{n+2}_1=\dots,\]
hence the sequence (\ref{ideals_of_FXM}) is stationary. Hence $F[X;M]$ is ACCP.
\end{proof}

The next theorem is our example of a monoid domain which is not an ACCP domain.

\begin{theorem}\label{thm_example}
Let 
\[p_1=2,\,\, p_2=3,\,\, p_3=5,\,\dots\]
be the sequence of prime numbers. Consider the submonoid 
\[M=\langle\, \frac{1}{p_1p_3},\; \frac{1}{p_2p_4},\; \frac{1}{p_3p_5}, \dots \,\rangle\]
of $\Q_+$. Then the monoid domain $F[X;M]$ is atomic, but not ACCP.
\end{theorem}

\begin{proof}
To show that $M$ is atomic, it is enough to show that each $\displaystyle{\frac{1}{p_ip_{i+2}}}$ ($i\ge 1$) is an atom of $M$. Suppose to the contrary, i.e., that $p_ip_{i+2}$ is not an atom for some $i\in\{1,2,3,\dots\}$. Then
\[\frac{1}{p_ip_{i+2}}=a+b,\]
where $a,b$ are two nonzero elements of $M$, both smaller than $\displaystyle{\frac{1}{p_ip_{i+2}}}$. Hence 
\begin{align}
a &= \frac{k_1}{p_{\alpha_1}p_{\alpha_1+2}}+\dots+\frac{k_r}{p_{\alpha_r}p_{\alpha_r+2}},\label{eq_1}\\
b &= \frac{l_1}{p_{\beta_1}p_{\beta_1+2}}+\dots+\frac{l_s}{p_{\beta_s}p_{\beta_s+2}},\label{eq_2}
\end{align}
where $i<\alpha_1<\dots <\alpha_r$ ($k_1,\dots, k_r\in\N$) and $i<\beta_1<\dots <\beta_s$ ($l_1,\dots, l_s\in\N$).
However when we add all the fractions from (\ref{eq_1}) and (\ref{eq_2}), using the product of the denominators that appear in (\ref{eq_1}) and (\ref{eq_2}) as a common denominator, we do not get $p_i$ in the denominator of the sum. That is a contradiction, so each $\displaystyle{\frac{1}{p_ip_{i+2}}}$ is an atom.

\smallskip
Now consider the following sequence of principal ideals of $M$:
\[\left(\frac{1}{2}+\frac{1}{3}\right), \left(\frac{1}{3}+\frac{1}{5}\right), \left(\frac{1}{5}+\frac{1}{7}\right), \dots, \left(\frac{1}{p_i}+\frac{1}{p_{i+1}}\right),\dots \]
We have
\[\left(\frac{1}{p_i}+\frac{1}{p_{i+1}}\right)\subseteq \left(\frac{1}{p_{i+1}}+\frac{1}{p_{i+2}}\right)\]
for $i=1,2,3,\dots$ since 
\[\frac{1}{p_i}+\frac{1}{p_{i+1}}=\frac{1}{p_{i+1}}+\frac{1}{p_{i+2}}+\frac{p_{i+2}-p_i}{p_ip_{i+2}}.\]
This sequence is not stationary, so $M$ is not ACCP. Hence, by Proposition \ref{ACCP_iff_ACCP}, $F[X;M]$ is not ACCP.

We now show that $F[X;M]$ is atomic. Suppose to the contrary. Then there is an element $f\in F[X;M]$ such that 
\begin{align*}
f &= f_0f_1\\
   &= f_0f_{10}f_{11}\\
   &= f_0f_{10}f_{110}f_{111}\\
   &=\dots,
\end{align*}
where $f_0, f_1, f_{10}, f_{11}, f_{110}, f_{111}, \dots$ are non-constant polynomials. Let $\beta$ be the largest exponet of $f$ and $\beta_{111\dots 10}$ (resp. $\beta_{111\dots 11}$) the largest exponents of the polynomials $f_{111\dots 10}$ (resp. $f_{111\dots 11}$). Since
\begin{align*}
\beta &= \beta_0+\beta_1\\
          &= \beta_0+\beta_{10}+\beta_{11}\\
          &= \beta_0+\beta_{10}+\beta_{110}+\beta_{111}\\ 
          &= \dots,
\end{align*}
no atom of $M$ can appear as an addend in infinitely many $\beta_{111\dots 10}$. Let $p_i$ be the largest prime that appears in the denominator of $\beta$, written in reduced form. There is an $n$ big enough so that when we write $\beta_0+\beta_{10}+\beta_{110}+\dots +\beta_{111\dots 10}$, with $n$ ones in the last index, as a sum of atoms, the atom 
$\displaystyle{\frac{1}{p_{j-2} p_j}}$ with the largest prime $p_j$ has $j>i$ and appears in the sum less than $p_j$ times (otherwise the sums $\beta_0+\beta_{10}+\beta_{110}+\dots +\beta_{111\dots 10}$ would become arbitrarily large). Then this sum of atoms, when written as a fraction in reduced form, would have $p_j$ in the denominator and so it could not be equal to $\beta$, a contradiction. Thus $F[X;M]$ is atomic.
\end{proof}

\section{Concluding remarks}
Let $M$ be a submonoid of $\Q_+$ and $F$ a field. If the monoid domain $F[X;M]$ is atomic, then the monoid $M$ is atomic, as it was proved in \cite[Proposition 2.10]{gk}. We would like to mention Question 2.11 from \cite{gk} which asks if the opposite direction is also true. In other words, is it true that $F[X;M]$ is atomic if and only if $M$ is atomic. (The analogous statement for the ACCP property is true, as we proved in the above Proposition \ref{ACCP_iff_ACCP}.) In our example in Theorem \ref{thm_example} both the monoid $M$ and the monoid domain $F[X;M]$ are atomic.

\end{document}